\documentclass{article}
\usepackage{amssymb,amsmath,amsthm,graphicx}

\def\qed{\hfill {\hbox{${\vcenter{\vbox{               
   \hrule height 0.4pt\hbox{\vrule width 0.4pt height 6pt
   \kern5pt\vrule width 0.4pt}\hrule height 0.4pt}}}$}}}

\newtheorem{theorem}{Theorem}

\newtheorem{corollary}[theorem]{Corollary}

\theoremstyle{definition}
\newtheorem{example}{Example}
\newtheorem{definition}{Definition}

\date{}

\title{\Large \textbf{Virtual Tribrackets}}

\author{Sam Nelson\footnote{Email: Sam.Nelson@cmc.edu. Partially supported by Simons Foundation collaboration grant 316709.}\and
Shane Pico\footnote{Email: SPico18@students.claremontmckenna.edu}}

\begin{document}
\maketitle

\begin{abstract}
We introduce \textit{virtual tribrackets}, an algebraic structure for coloring
regions in the planar complement of an oriented virtual knot or link diagram. 
We use these structures to define counting invariants of virtual knots and links
and provide examples of the computation of the invariant; in particular we 
show that the invariant can distinguish certain virtual knots.
\end{abstract}

\parbox{6in} {\textsc{Keywords:} biquasiles, virtual biquasiles, tribrackets, 
virtual tribrackets, virtual knots and links

\smallskip

}

\section{\large\textbf{Introduction}}\label{I}

In \cite{dsn2}, algebraic structures known as \textit{biquasiles} were
introduced and used to define invariants of knots and links via vertex
colorings of \textit{dual graph diagrams} which uniquely determine
oriented knot diagrams on the sphere. In \cite{cnn}, biquasile counting 
invariants were enhanced with Boltzmann weights analogously to quandle
2-cocycle invariants, and in \cite{KN} biquasile colorings and Boltzmann 
enhancements were used to distinguish oriented surface-links in 
$\mathbb{R}^4$.

In \cite{MN}, \textit{ternary quasigroups} were used to define invariants 
of knots and links. While the details of the relationship between ternary 
quasigroups and biquasiles are currently under investigation, it is easy
to see that a biquasile defines a ternary quasigroup and that ternary
quasigroups satisfying certain conditions give rise to biquasiles.

Virtual knots and links were introduced in \cite{K} and have been much
studied in the years since \cite{KB}. 
As detailed in \cite{CKS,KK,K} etc., a virtual knot can be understood
geometrically as an equivalence class of knots in thickened orientable
surfaces up to stabilization, i.e. adding or removing handles in the complement
of the knot on the surface. One fundamental difference between coloring knots
with tribrackets as opposed to \textit{a prior} similar 
coloring structures such as quandles and biquandles arise when we consider the
case of virtual knots and links. For quandles and biquandles, where the colors
are attached to the arcs or semiarcs, we can simply ignore virtual crossings
with no problem. Any attempt to do this with tribrackets by coloring regions
in the surface complement of a virtual knot is rendered impossible by 
stabilization: any two regions can be connected by a handle, requiring the 
same color.

\[\scalebox{0.8}{\includegraphics{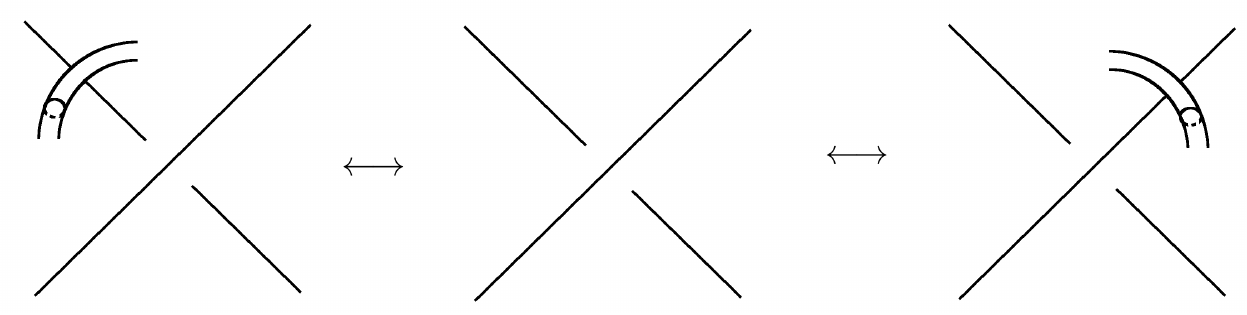}}\]

For this reason, we must include a second tribracket operation at the virtual 
crossings, analogous to \textit{virtual biquandles} defined in \cite{KM},
algebraic structures motivated by virtual Reidemeister moves with distinct 
operations at virtual crossings and at classical crossings.
In this paper we introduce \textit{virtual tribrackets}, a coloring structure
for oriented virtual knots and links generalizing biquasiles and ternary 
quasigroups by defining distinct operations at classical and virtual crossings.
The paper is organized as follows. In Section \ref{VT} we
define virtual tribrackets and identify some examples. In Section \ref{VK}
we use virtual tribrackets to define new counting invariants of virtual 
knots and links. We provide examples to illustrate the method of computation
of the invariant and compute the invariant via \texttt{python} code for 
the virtual knots in the table at \cite{KA} for a few example virtual 
tribrackets. In Section \ref{Q} we conclude with some open questions for 
future research.

\section{\large\textbf{Virtual Tribrackets}}\label{VT}

We begin with a definition adapted from an equivalent definition in \cite{MN}.

\begin{definition}
Let $X$ be a set. A \textit{(vertical) Niebrzydowski tribracket} on $X$ is 
a ternary operation $[,,]:X\times X\times X\to X$ satisfying the conditions
\begin{itemize}
\item[(i)] In the equation
\[\begin{array}{rcl}
[a,b,c] & = & d \\
\end{array}\] 
any three of the four elements $\{a,c,b,d\}$ determine the fourth.
\item[(ii)] For all $a,b,c,d\in X$ we have
\[\begin{array}{rcll}
{[}a,b,[b,c,d]] & = & [a,[a,b,c],[[a,b,c],c,d]] & (iii.i) \\
{[}[a,b,c],c,d] & = & [[a,b,[b,c,d]],[b,c,d],d] & (iii.ii).
\end{array}\]
\end{itemize}
For brevity we will refer to this structure as a \textit{tribracket}.
\end{definition}

The motivation for the tribracket axioms is to make the set of colorings 
of the regions in the planar complement of an oriented knot or link diagram
using the coloring rule below correspond bijectively before and after 
Reidemeister moves.
\[\includegraphics{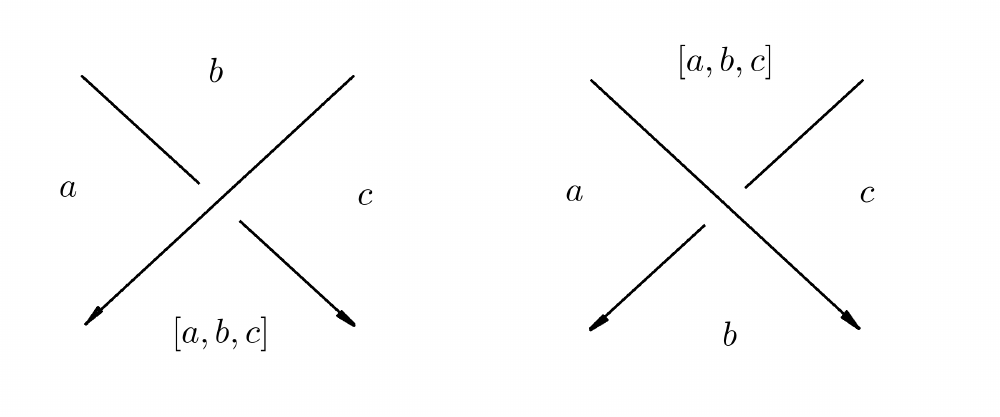}\]
The existence and uniqueness of colorings on both sides of Reidemeister
I and II moves is satisfied by condition (i), while condition (ii) follows 
from the oriented Reidemeister III move below, which together with the
eight possible oriented Reidemeister I and II moves which form a 
generating set of oriented Reidemeister moves.
\[\includegraphics{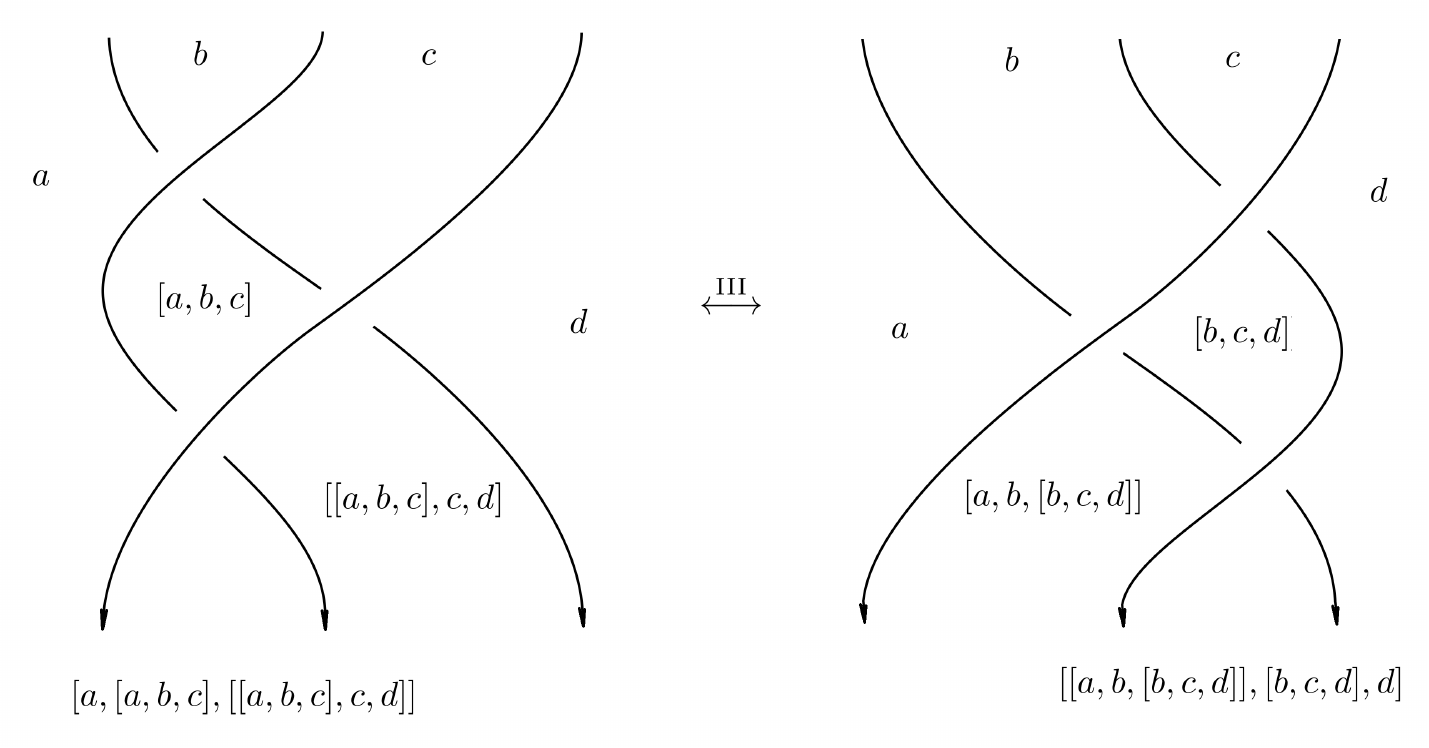}\]

We then have the following result (see also \cite{dsn2,MN}):
\begin{theorem}
Let $X$ be a set with a tribracket. The number $\Phi^{\mathbb{Z}}_X$ of 
tribracket colorings of an oriented knot diagram is an integer-valued 
invariant of oriented knots.
\end{theorem}

\begin{example}
Let $R$ be any commutative ring and let $x,y\in R^{\times}$ be units in $R$.
Then the ternary operation 
\[[a,b,c]=xa-xyb+yc\]
defines an tribracket called an \textit{Alexander tribracket}. 
These are related to the \textit{Alexander Biquasiles} defined in \cite{dsn2} 
by $x=dn$ and $y=sn$.
\end{example}

In \cite{K}, classical knot theory was extended to \textit{virtual knot theory}
containing classical knot theory as a subset. A virtual knot diagram includes
classical crossings as well as \textit{virtual crossings}, drawn as a circled 
transverse intersections with no over or under information. These virtual
crossings can be regarded as representing genus in the supporting surface 
of the knot, which is defined up to stabilization moves on the surface 
\cite{KK}. A \textit{virtual knot} is an equivalence class 
of virtual knot diagrams under the following \textit{virtual Reidemeister 
moves}, together with the classical Reidemeister moves:

\[\includegraphics{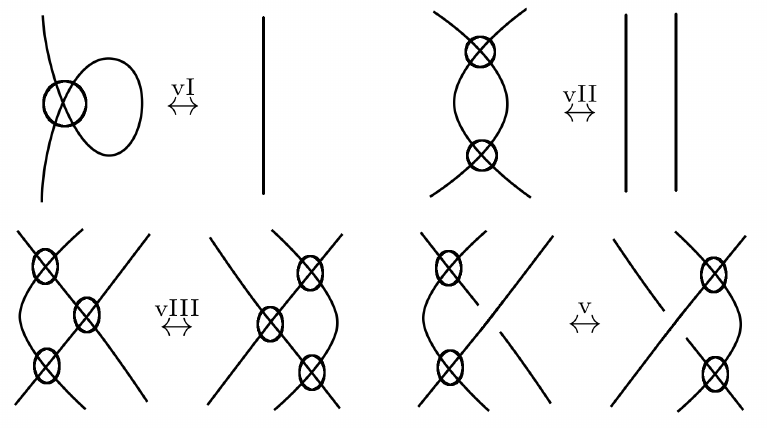}\]

We will be interested in oriented virtual knots and links. It is 
straightforward to show that the oriented moves below, together with 
the previously mentioned classical Reidemeister moves, form 
a generating set for the set of all oriented virtual moves.

\[\includegraphics{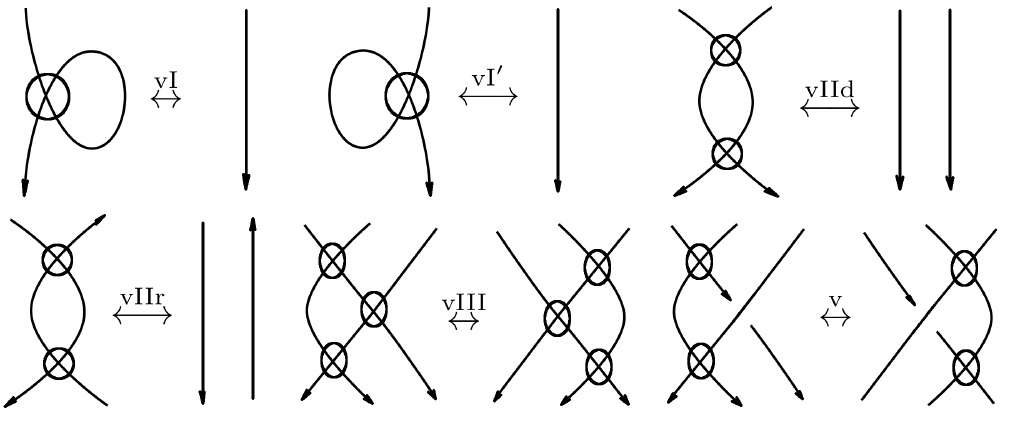}\]

With these moves in mind, we will define a \textit{virtual tribracket}
structure with following coloring rules:

\[\includegraphics{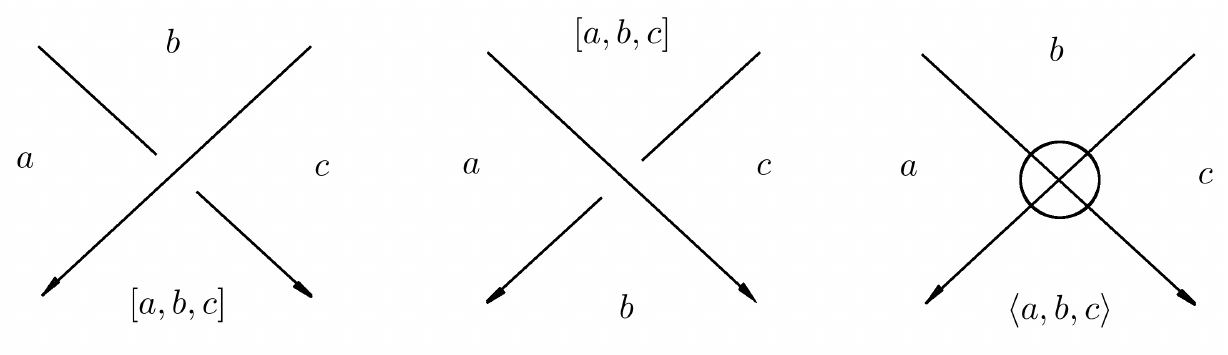}\]

\begin{definition}
Let $X$ be a set. A \textit{virtual tribracket} structure on $X$ is a pair of
ternary operations $[,,]:X\times X\times X\to X$ and
$\langle,,\rangle:X\times X\times X\to X$ satisfying the conditions
\begin{itemize}
\item[(i)] In the equations
\[[a,b,c]  =  d \quad\mathrm{and} \quad \langle a,b,c\rangle  = d 
\] 
any three of the four elements $\{a,c,b,d\}$ determine the fourth.
\item[(ii)] For all $a,b,c\in X$ we have
\[\begin{array}{ccll}
\langle a,\langle a,b,c\rangle, c\rangle & = &  b & (vii)\\
{[}a,b,[b,c,d]] & = & [a,[a,b,c],[[a,b,c],c,d]] & (iii.i)\\
{[}[a,b,c],c,d] & = & [[a,b,[b,c,d]],[b,c,d],d] & (iii.ii)\\
\langle a,b,\langle b,c,d\rangle\rangle 
& = & \langle a,\langle a,b,c\rangle,\langle \langle a,b,c\rangle,c,d\rangle\rangle & (viii.i) \\
\langle \langle a,b,c\rangle,c,d\rangle 
& = & \langle\langle a,b,\langle b,c,d\rangle\rangle,\langle b,c,d\rangle,d\rangle & (viii.ii)\\
{[} a,b,\langle b,c,d\rangle] 
& = & \langle a,\langle a,b,c\rangle,[ \langle a,b,c\rangle,c,d ]\rangle & (v.i) \\
{[} \langle a,b,c\rangle,c,d ] & = & \langle [ a,b,\langle b,c,d]\rangle,\langle b,c,d\rangle,d\rangle & (v.ii)\\
\end{array}\]
\end{itemize}
\end{definition}

We then have

\begin{theorem}\label{main}
A virtual tribracket coloring of an oriented virtual link diagram before 
an oriented classical or virtual Reidemeister move determines a unique
coloring of the diagram after the move agreeing outside the neighborhood 
of the move.
\end{theorem}

\begin{proof}
The case of classical moves can be found in \cite{dsn2, MN}. The cases of
moves vI, vI$'$ and vIII are analogous to the classical cases. Hence
it remains only to consider the cases of moves vIId, vIIr and v. 
Axioms (i) and  ($vii$) together satisfy moves vIId and vIIr:
\[\includegraphics{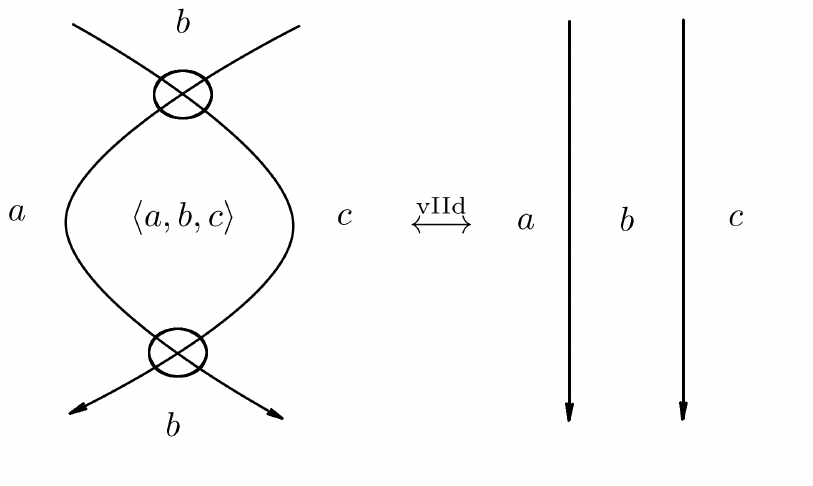}\]
\[\includegraphics{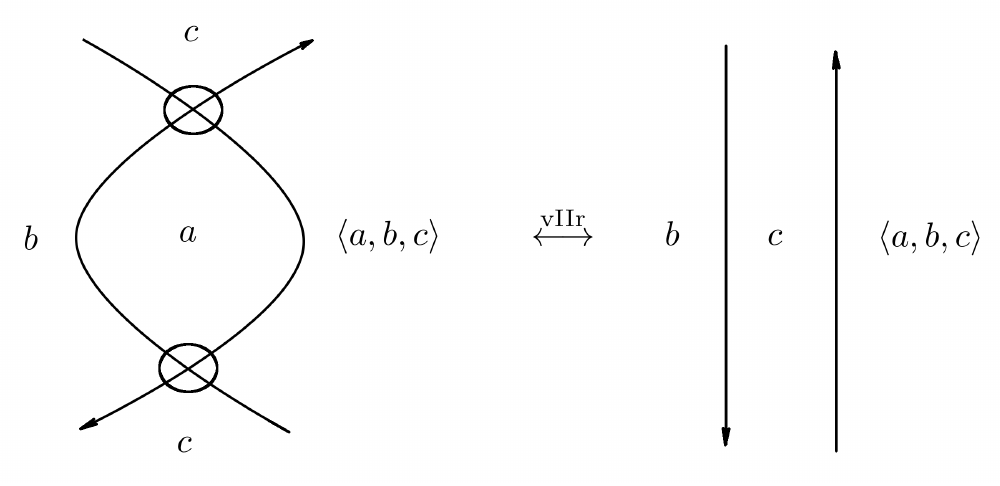}\]
while axioms ($v.i$) and $(v.ii)$ satisfy move v:
\[\includegraphics{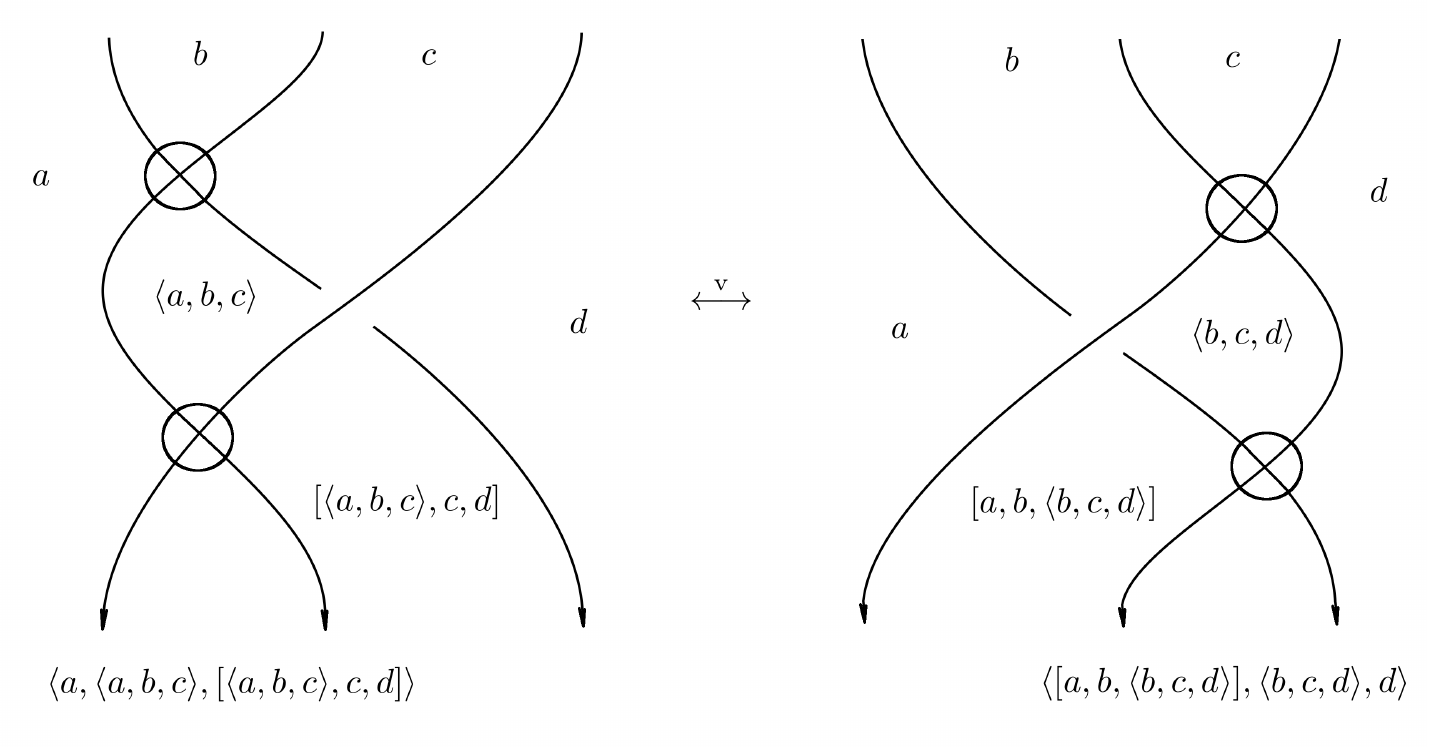}\]
\end{proof}

\begin{example}
Let $X$ be an Alexander tribracket structure on a commutative ring $R$
given by 
\[[a,b,c]=xy-xyb+yc\]
and let $v$ be a unit in $R$ satisfying $1+xy=v^{-1}x+vy$. We can 
give $X$ the structure of a virtual tribracket by setting
\[\langle a,b,c\rangle=va-b+v^{-1}c.\]
Then verifying that our definition satisfies the axioms, we have
\begin{eqnarray*}
\langle a,\langle a,b,c\rangle, c\rangle 
& = & va-(va-b+v^{-1}c)+v^{-1}c \\
& = & va-va+b-v^{-1}c+v^{-1}c \\
& = & b
\end{eqnarray*}
so axiom (ii) is satisfied. Then
\begin{eqnarray*}
\langle a,b,\langle b,c,d\rangle\rangle 
& = & va-b+v^{-1}(vb-c+v^{-1}d) \\
& = & va+0b-c+v^{-1}d \\
& = & va-(va-b+v^{-1}c)+v^{-1}(v(va-b+v^{-1}c)-c+v^{-1}d) \\
& = & \langle a,\langle b,c,d\rangle,\langle\langle b,c,d\rangle,c,d\rangle \rangle \\
\end{eqnarray*}
and
\begin{eqnarray*}
\langle\langle a,b,c\rangle,c,d\rangle 
& = & v(va-b+v^{-1}c) -c+v^{-1}d \\
& = & v^2 a-vb+0c+v^{-1}d \\
& = & v(va-b+v^{-1}(vb-c+v^{-1}d))-(vb-c+v^{-1}d)+v^{-1}d \\
& = & \langle \langle a,b,\langle b,c,d \rangle\rangle,\langle b,c,d\rangle,d\rangle \\
\end{eqnarray*}
so axioms (viii.i) and (viii.ii) are satisfied. Finally, we check axioms 
(v.i) and (v.ii), keeping in mind that $1+xy=v^{-1}x+vy$:
\begin{eqnarray*}
[a,b,\langle b,c,d\rangle] 
& = & xa-xyb+y(vb-c+v^{-1}d) \\
& = & (v-v+x)a+(1-xv^{-1})b+(-v^{-1}+xv^{-2}-xyv^{-1})c+(v^{-1}y)d \\
& = & va-(va-b+v^{-1}c)+v^{-1}(x(va-b+v^{-1}c)-xy c+yd) \\
& = & \langle a, \langle a,b,c\rangle, [\langle a,b,c\rangle,c,d]\rangle \\
\end{eqnarray*}
and 
\begin{eqnarray*}
[\langle a,b,c\rangle,c,d] 
& = & x(va-b+v^{-1}c)-xyc+yd  \\
& = & (xv)a+(-xyv+yv^2-v)b+(1-yv)c+(v^{-1}+y-v^{-1})d \\
& = & v(xa-xyb+y(vb-c+v^{-1}d))-(vb-c+v^{-1}d)+v^{-1}d \\
& = & \langle [a,b,\langle b,c,d\rangle],\langle b,c,d\rangle, d\rangle \\
\end{eqnarray*}
and axioms (v.i) and (v.ii) are satisfied. 
\end{example}

We can define virtual tribracket structures on finite sets without
formulas by encoding the operation tables of the ternary operations
$[,,]$ and $\langle,,\rangle$
as 3-tensors, i.e., as $n$-component vectors of $n\times n$
matrices. More precisely, let $X=\{1,2,3,\dots, n\}$. Then
we will specify a virtual tribracket structure on $X$
by giving two ordered $n$ component lists of $n\times n$ matrices with
the property that to find $[a,b,c]$ we look at the first list,
find the $a$th matrix, and look up the entry in row $b$ column $c$;
to find $\langle a,b,c\rangle$, we do the same with the second list.

\begin{example}
The pair of 3-tensors below defines a virtual tribracket structure on
the set $X=\{1,2,3\}$:
\[\left[
\left[\begin{array}{rrr} 2 & 1 & 3\\ 3 & 2 & 1\\ 1 & 3 & 2\\\end{array}\right], 
\left[\begin{array}{rrr} 3 & 2 & 1\\ 1 & 3 & 2\\ 2 & 1 & 3\\\end{array}\right],
\left[\begin{array}{rrr} 1 & 3 & 2\\ 2 & 1 & 3\\ 3 & 2 & 1\\\end{array}\right]\right],\left[
\left[\begin{array}{rrr} 1 & 2 & 3\\ 3 & 1 & 2\\ 2 & 3 & 1\\\end{array}\right],
\left[\begin{array}{rrr} 2 & 3 & 1\\ 1 & 2 & 3\\ 3 & 1 & 2\\\end{array}\right],
\left[\begin{array}{rrr} 3 & 1 & 2\\ 2 & 3 & 1\\ 1 & 2 & 3\\\end{array}\right]\right]
\]
Then we have for instance $[1,3,2]=3$ and $\langle 2,3,3\rangle=2$.
\end{example}

\section{\Large\textbf{Invariants from Virtual Tribrackets}}\label{VK}

We begin this section with a corollary of theorem \ref{main}.

\begin{corollary}
The number of colorings of an oriented virtual link diagram by a virtual 
tribracket $X$ is an integer-valued invariant of virtual links.
\end{corollary}

\begin{definition}
We will denote the number of $X$-colorings of an oriented virtual link
diagram $D$ representing an oriented virtual link $K$ for a virtual
tribracket $X$ by $\Phi^{\mathbb{Z}}_X(K)$.
\end{definition}

\begin{example}
The trivial value of $\Phi^{\mathbb{Z}}_X(K)$ on an oriented 
virtual link of $n$ components is $|X|^{n+1}$ since an unlink of $n$
components with no crossings divides the plane into $n+1$ regions with
no relations between the colors.
\end{example}

\begin{example}
Let $X=\mathbb{Z}_3$ and set $x=1, y=2$ and $v=2$. Then we observe that
\[1+xy=1+2=0=1(2)+2(2)=xv^{-1}+yv\]
and we have a virtual Alexander tribracket with
\begin{eqnarray*}
[a,b,c] & = & \langle a,b,c\rangle = a +b+2c\\
\langle a,b,c\rangle & = & 2a+2b+2c.
\end{eqnarray*} 
Let us compute the set of $X$-colorings of the virtual knot 3.7 below.
\[\includegraphics{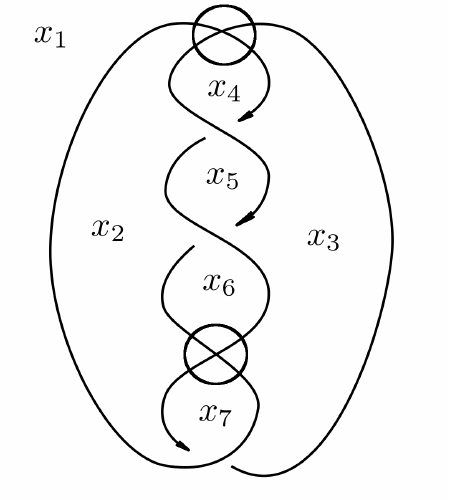}\]
We obtain a homogeneous system of linear equations over $\mathbb{Z}_3$ 
with coefficient matrix
\[\left[\begin{array}{rrrrrrr}
2 & 2 & 2 & 2 & 0 & 0 & 0 \\
0 & 1 & 2 & 2 & 1 & 0 & 0 \\
0 & 1 & 2 & 0 & 2 & 1 & 0 \\
0 & 2 & 2 & 0 & 0 & 2 & 2 \\
2 & 1 & 2 & 0 & 0 & 0 & 1
\end{array}\right]
\stackrel{\mathrm{row\ moves\ over\ }\mathbb{Z}_3}{\longleftrightarrow}
\left[\begin{array}{rrrrrrr}
1 & 1 & 1 & 1 & 0 & 0 & 0 \\
0 & 1 & 2 & 2 & 1 & 0 & 0 \\
0 & 0 & 1 & 2 & 1 & 2 & 2 \\
0 & 0 & 0 & 1 & 1 & 1 & 0 \\
0 & 0 & 0 & 0 & 0 & 0 & 0 \\
\end{array}\right]\]
so the kernel is 3-dimensional and there are $\Phi^{\mathbb{Z}}_X(3.7)=3^3=27$
colorings. This distinguishes this virtual knot from the unknot, which has 
$\Phi^{\mathbb{Z}}_X(0_1)=2^2=9$ colorings.
\end{example}

\begin{example}
Unlike the \textit{a priori} similar case of quandle colorings, 
$\Phi^{\mathbb{Z}}_X(L)$ need not be nonzero. For example, the Hopf link
\[\includegraphics{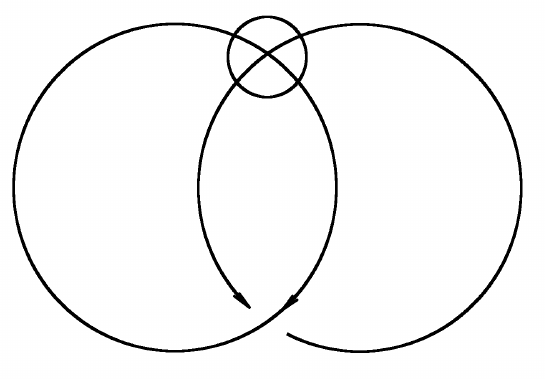}\]
has no colorings by the virtual tribracket with operation tensors
\[
\left[\left[
\begin{array}{rrr}
1& 2& 3\\
3& 1& 2\\
2& 3& 1
\end{array}\right],\left[\begin{array}{rrr}
2& 3& 1\\
1& 2& 3\\
3& 1& 2
\end{array}\right],\left[\begin{array}{rrr}
3& 1& 2\\
2& 3& 1\\
1& 2& 3
\end{array}\right]\right],
\left[\left[\begin{array}{rrr}
2& 3& 1\\
1& 2& 3\\
3& 1& 2
\end{array}\right],\left[\begin{array}{rrr}
3& 1& 2\\
2& 3& 1 \\
1& 2& 3
\end{array}\right],\left[\begin{array}{rrr}
1& 2& 3\\
3& 1& 2 \\
2& 3& 1
\end{array}\right]\right].
\]
This result distinguishes this virtual link from the unlink, which has
$3^3=27$ colorings by $X$.
\end{example}

\begin{example}
We computed $\Phi_X^{\mathbb{Z}}(K)$ for the prime virtual links with up to seven 
crossings in the table in \cite{KA} using several virtual tribracket 
structures with our \texttt{python} code. The virtual tribracket given the 
by 
\[
\left[
\left[\begin{array}{rrr}
1& 2& 3\\
2& 3& 1\\
3& 1& 2\end{array}\right],
\left[\begin{array}{rrr}
3& 1& 2\\
1& 2& 3\\
2& 3& 1\end{array}\right],
\left[\begin{array}{rrr}
2& 3& 1\\
3& 1& 2\\
1& 2& 3\end{array}\right]\right],\left[
\left[\begin{array}{rrr}
3& 2& 1\\
2& 1& 3\\
1& 3& 2\end{array}\right],
\left[\begin{array}{rrr}
2& 1& 3\\
1& 3& 2\\
3& 2& 1\end{array}\right],
\left[\begin{array}{rrr}
1& 3& 2\\
3& 2& 1\\
2& 1& 3\end{array}\right]\right].
\]
distinguishes the virtual knots $3.6, 3.7,4.69,4.70,4.71,4.72,4.73,4.74,4.75,4.76, 4.77, 4.98$ and $4.99$ with $27$ colorings from the unknot with $9$ 
colorings , while the the virtual tribracket given by
\[
\left[
\left[\begin{array}{rrrr}
3& 1& 2& 4\\
1& 3& 4& 2\\
2& 4& 3& 1\\
4& 2& 1& 3
\end{array}\right],
\left[\begin{array}{rrrr}
4& 2& 1& 3\\
2& 4& 3& 1\\
1& 3& 4& 2\\
3& 1& 2& 4
\end{array}\right],
\left[\begin{array}{rrrr}
1& 3& 4& 2\\
3& 1& 2& 4\\
4& 2& 1& 3\\
2& 4& 3& 1
\end{array}\right],
\left[\begin{array}{rrrr}
2& 4& 3& 1\\
4& 2& 1& 3\\
3& 1& 2& 4\\
1& 3& 4& 2
\end{array}\right]\right],\]\[
\left[
\left[\begin{array}{rrrr}
1& 2& 3& 4\\
2& 1& 4& 3\\
3& 4& 1& 2\\
4& 3& 2& 1
\end{array}\right],
\left[\begin{array}{rrrr}
2& 1& 4& 3\\
1& 2& 3& 4\\
4& 3& 2& 1\\
3& 4& 1& 2
\end{array}\right],
\left[\begin{array}{rrrr}
3& 4& 1& 2\\
4& 3& 2& 1\\
1& 2& 3& 4\\
2& 1& 4& 3
\end{array}\right],
\left[\begin{array}{rrrr}
4& 3& 2& 1\\
3& 4& 1& 2\\
2& 1& 4& 3\\
1& 2& 3& 4
\end{array}\right]\right]
\]
distinguishes the virtual knots $3.6,4.65,4.69,4.98,4.102,4.104$ and $4.108$
with $64$ colorings from the unknot with $16$ colorings.
\end{example}

\begin{example}
The counting virtual tribracket counting invariant is sensitive to orientation 
reversal. Consider the two virtual links $L$ and $L'$ which differ only in 
the orentation of one component:
\[\includegraphics{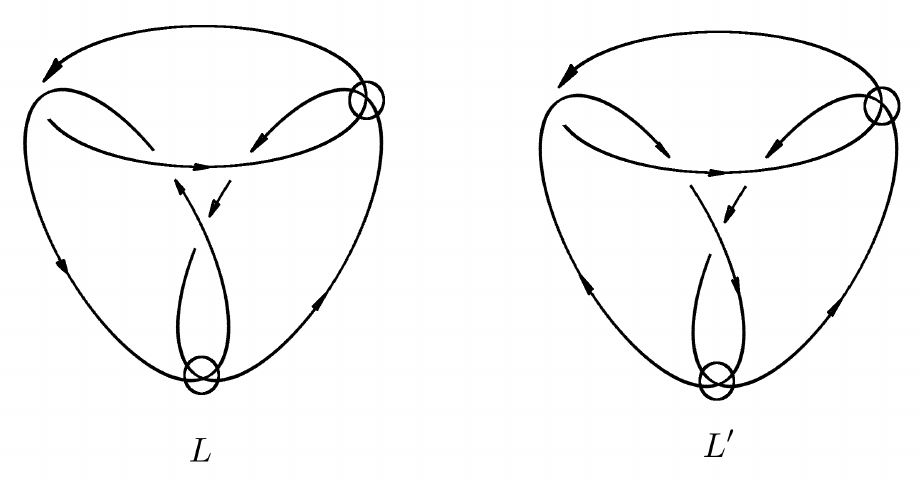}\]
Our \texttt{python} computations indicate that the virtual tribracket
structure on $X=\{1,2,3,4\}$ defined by the operation tensor
\[
\left[\left[\begin{array}{rrrr} 
1 & 3 & 4 & 2 \\4 & 2 & 1 & 3  \\ 2 & 4 & 3 & 1 \\ 3 & 1 & 2 & 4
\end{array}\right],\left[\begin{array}{rrrr}
3 & 1 & 2 & 4 \\2 & 4 & 3 & 1 \\4 & 2 & 1 & 3 \\1 & 3 & 4 & 2
\end{array}\right],\left[\begin{array}{rrrr}
4 & 2 & 1 & 3 \\1 & 3 & 4 & 2 \\3 & 1 & 2 & 4 \\2 & 4 & 3 & 1
\end{array}\right],\left[\begin{array}{rrrr}
2 & 4 & 3 & 1 \\3 & 1 & 2 & 4 \\1 & 3 & 4 & 2 \\4 & 2 & 1 & 3,
\end{array}\right]\right],\]\[
\left[\left[\begin{array}{rrrr}
1 & 4 & 2 & 3 \\2 & 3 & 1 & 4 \\3 & 2 & 4 & 1 \\4 & 1 & 3 & 2
\end{array}\right],\left[\begin{array}{rrrr}
3 & 2 & 4 & 1 \\4 & 1 & 3 & 2 \\1 & 4 & 2 & 3 \\2 & 3 & 1 & 4
\end{array}\right],\left[\begin{array}{rrrr}
4 & 1 & 3 & 2 \\3 & 2 & 4 & 1 \\2 & 3 & 1 & 4 \\1 & 4 & 2 & 3
\end{array}\right],\left[\begin{array}{rrrr}
2 & 3 & 1 & 4 \\1 & 4 & 2 & 3 \\4 & 1 & 3 & 2 \\3 & 2 & 4 & 1
\end{array}\right]\right]
\]
distinguishes the links with counting invariant values of
$\Phi_X^{\mathbb{Z}}(L)=16\ne \Phi_X^{\mathbb{Z}}(L')=64.$
\end{example}

\section{Questions}\label{Q}

As with all counting invariants, we can ask about enhancements. In \cite{cnn}
and \cite{MN2}, Boltzmann weights are used to enhance the biquasile counting
invariant. A scheme similar to that used in \cite{CN} could be applied 
similarly here with different Boltzmann weights at the classical and virtual
crossings to obtain a two-variable polynomial enhancements of the virtual 
tribracket counting enhancement.

Another avenue of future research could be using Alexander virtual tribrackets
to define an analog of the Sawollek polynomial coming for the virtual tribracket
as opposed to from the virtual biquandle as in \cite{S}. 
Are these polynomials the same?

\bibliography{sn-sp2}{}
\bibliographystyle{abbrv}

\bigskip

\noindent
\textsc{Department of Mathematical Sciences \\
Claremont McKenna College \\
850 Columbia Ave. \\
Claremont, CA 91711}

\end{document}